\documentclass{amsart}

\usepackage{epsfig}
\usepackage{amsmath}
\usepackage{amssymb}
\usepackage{amscd}
\usepackage{graphicx}
\usepackage{color}
\usepackage{verbatim}
\usepackage{dsfont}
\usepackage[symbol*]{footmisc}

\DefineFNsymbols*{asterisks}{{$\star$}{$\dagger$}{$\ddagger$}{$\dagger\dagger$}{$\ddagger\ddagger$}}
\setfnsymbol{asterisks}

\newtheorem{thm}{Theorem}[section]
\newtheorem{lem}[thm]{Lemma}

\newtheorem{fact}[thm]{Fact}

\newtheorem{defn}[thm]{Definition}

\theoremstyle{remark}

\def \N {\mathbb N}

\def \G {\mathcal G}
\def \F {\mathcal F}

\def \R {\mathbb R}

\def \M {\mathcal M}
\def \P {\mathbb P}

\def \sq {sequence}
\def \xt {$(X,T)$}

\def \tl {topological}
\def \im {invariant measure}

\def \ds {dynamical system}

\def \mmu {\mu}

\numberwithin{equation}{section}

\title[A strictly ergodic subshift uncorrelated to the M\"obius function]{A strictly ergodic, 
positive entropy subshift uniformly uncorrelated to the M\"obius function}
       
\author{Tomasz Downarowicz and Jacek Serafin}

\address{Faculty of Pure and Applied Mathematics, Wroc{\l}aw University of Science and Technology,
Wybrze\.ze Wyspia\'nskiego 27, Wroc{\l}aw 50-370, Poland, \rm downar@pwr.edu.pl ,\break serafin@pwr.edu.pl .}


\thanks{The research is supported by the NCN (National Science Center, Poland) grant 2013/08/A/ST1/00275,
and by Wroc\l aw University of Science and Technology, grant 0401/0155/18.}

\begin{document}
\begin{abstract}
A recent result \cite{DS} shows that there exist positive entropy subshifts 
satisfying the assertion of Sarnak's conjecture \cite{sarnak}.
More precisely, it is proved that if $y=(y_n)_{n\ge 1}$ 
is a bounded sequence with zero average along every infinite arithmetic progression 
(the M\"obius function is an example of such a \sq\ $y$) then for every $N\ge 2$ 
there exists a subshift $\Sigma$ over $N$ symbols, with entropy arbitrarily close to 
$\log N$, uncorrelated to $y$.  

In the present note, we improve the result of \cite{DS}. First of all, we observe that the uncorrelation obtained in \cite{DS} is \emph{uniform}, i.e., for any continuous function $f:\Sigma\to\R$ and every $\epsilon>0$ there exists $n_0$ such that for any $n\ge n_0$ and any $x\in\Sigma$ we have
$$
\left|\frac1n\sum_{i=1}^{n}f(T^ix)\,y_i\right|<\epsilon.
$$
More importantly, by a fine-tuned modification of the construction from \cite{DS}
we create a \emph{strictly ergodic} subshift, with all the desired properties of the example in \cite{DS} (uniformly uncorrelated to $y$ and with entropy arbitrarily close to $\log N$).

The question about these two additional properties (uniformity of uncorrelation and strict ergodicity) has been posed by Mariusz Lema\' nczyk in the context of the so-called strong MOMO (M\"obius Orthogonality on Moving Orbits) property. Our result shows, among other things, that strong MOMO is essentially stronger than uniform uncorrelation, even for strictly ergodic systems.
\end{abstract}
\maketitle

\numberwithin{equation}{section}

\section{Preliminaries}
Let $y$ be a bounded, real-valued \sq\ with zero average along every infinite arithmetic progression, 
i.e., satisfying, for every $t\ge 1$ and $l\ge 0$, the condition 
\begin{equation}\label{arp}
\lim_n \frac1n\sum_{i=1}^n y_{it+l} = 0.
\end{equation}
A sequence as above we call \emph{aperiodic}. Clearly we may (and will) assume that $|y_n|\le 1$ for all $n$. 
An important example of an aperiodic sequence is: 
$$
\mmu_n=\begin{cases}
\phantom{-}1&\text{for $n=1$,}\\
\phantom{-}(-1)^r&\text{if $n$ is a product of $r$ distinct primes,}\\
\phantom{-}0& \text{otherwise (i.e., if $n$ has a repeated prime factor),}
\end{cases}
$$
called the M\"obius function $y=\mmu$ (see e.g. \cite{sarnak}). 

By a \emph{\tl\ \ds} we will mean a pair \xt\ where $X$ is a compact metric space and $T:X\to X$ is a continuous transformation. Uncorrelation between a system and a \sq\ will be understood as follows:

\begin{defn}\label{def1}
We say that \xt\ is \emph{uncorrelated} to $y$ if for each continuous function $f:X\to\R$ and every $x\in X$, we have 
$$
\lim_n\frac 1n\sum_{i=1}^n f(T^ix)y_i = 0.
$$
If, for each fixed function $f$, the above convergence to $0$ is uniform on $X$ then we will say that \xt\ is \emph{uniformly uncorrelated} to $y$.
\end{defn}

Let $\Lambda$ be a fixed finite alphabet. By a \emph{subshift} $\Sigma$ we will mean any closed and shift-invariant subset of $\Lambda^{\N}$. A subshift becomes a \ds\ when regarded together with the action of the shift map. We denote by $\M(\Sigma)$ the collection of all shift-invariant Borel probability measures supported by $\Sigma$; this set is non-empty and compact in the weak-star topology with the distance between measures given by:
\begin{equation}\label{metr}
d(\mu, \nu)=\sum_{n=1}^{\infty}\frac{1}{2^n}\sum_{D\in\Lambda^n}|\mu(D)-\nu(D)|,
\end{equation}
where $\Lambda^n$ denotes the family of blocks of length $n$ over $\Lambda$, identified with the corresponding cylinder sets. Since, for any $n'<n$, any cylinder corresponding to a block of length $n'$ is a disjoint union of cylinders corresponding to blocks of length $n$, for proving that two measures $\mu$ and $\nu$ are close it suffices to find just one large enough integer $n$ and small enough $\theta>0$ such that $|\mu(D)-\nu(D)|<\theta$ for all blocks $D$ of length $n$.

If $D=(d_1,\dots,d_n)$ and $C=(c_1,\dots,c_n)$ are finite sequences (blocks) of real numbers of the same length, we define their correlation as the average
$$
\overline{DC}=\frac1n\sum_{i=1}^nd_ic_i.
$$
If $C=(c_1,\dots,c_n)$ and $D$ is a block of length $n'$ where $n'\leq n$ then the frequency of the occurrence of $D$ in $C$ is
$$
\mathsf{freq}_C(D):=\frac1n{\#\{i=1,\ldots,n-n'+1: c_i\ldots c_{i+n'-1}=D\}}.
$$

\section{The main result}

In \cite{DS} we have proved:
\begin{thm}\label{thm1}
Let $y$ be an aperiodic sequence and let $N\ge 2$ be a fixed integer. 
There exists a subshift $\Sigma$ over $N$ symbols of \tl\ entropy arbitrarily close to $\log N$, uncorrelated to $y$.
\end{thm}
A closer examination reveals that in fact we have obtained a stronger result: the above subshift $\Sigma$ is uniformly uncorrelated to $y$. We will justify this observation soon (see Fact \ref{fac} below).

The main goal of this short note is to strengthen the hypothesis and prove:
\begin{thm}\label{thm2}
Let $y$ be an aperiodic sequence and let $N\ge 2$ be a fixed integer. 
There exists a strictly ergodic subshift $\Sigma'$ over $N$ symbols of \tl\ entropy 
arbitrarily close to $\log N$, uniformly uncorrelated to $y$.
\end{thm}

The motivation for the above refinement of the former result from \cite{DS} comes partly from the so-called strong MOMO property, introduced by Lema\'nczyk and coauthors in a recent paper \cite{AKLR}. For a topological dynamical system $(X,T)$ the strong MOMO property is a form of disjointness between the system and the M\"obius function. If the M\"obius function is replaced by a bounded \sq\ $y$, we are dealing with an analog of the strong MOMO property, which we will call ``strong $y$-MOMO'' (although puristically it should be ``strong $y$OMO''). It is important that the strong $y$-MOMO property implies uniform uncorrelation between $(X,T)$ and $y$. The authors of \cite{AKLR} prove that Sarnak's conjecture is equivalent to its version in which M\"obius disjointness is replaced by the strong MOMO property. It is also proved that if $y$ is generic for a Bernoulli measure then no system of positive entropy has the strong $y$-MOMO property. On the other hand, using a disjointess argument, one can show that every zero entropy system has the strong $y$-MOMO property. In this manner, the strong $y$-MOMO property remarkably allows to ``distinguish zero from positive entropy'' using just one ``test \sq''.  Our current result implies that uniform uncorrelation to just one ``test sequence'' $y$ does not allow to fully distinguish between zero and positive entropy, hence uniform uncorrelation is (at least when $y$ is generic for a Bernoulli measure) essentially weaker than the strong $y$-MOMO property. 

Existence of a strictly ergodic example is important in the following context: in \cite{CDS} we have shown that it is very unlikely that every strictly ergodic system is uncorrelated to the M\"obius function. For that, the measure generated by the M\"obius function would have to be uncorrelated to any ergodic measure. This is a very specific and rare property which fails, for instance, if the Chowla conjecture holds. So, it is believed that the M\"obius function correlates with many strictly ergodic systems, and under Sarnak's conjecture, all such systems have positive entropy. In this setup, there are a priori three possibilites: 
\begin{enumerate}
	\item All strictly ergodic systems with positive entropy correlate with the M\"obius function;
	\item Every ergodic system with positive entropy has a strictly ergodic model which correlates with the M\"obius function;
	\item There exists an ergodic system with positive entropy whose all ergodic models do not correlate with the M\"obius function.
\end{enumerate}

Our result eliminates the option (1) above. Whether (2) or (3) is true, remains at the moment an open question. The option (2) leaves another possibility open: perhaps every ergodic system has a strictly ergodic model which does not correlate with the M\"obius function (our guess is that this is not true).

\section{The proofs}
We need to start by briefly recalling some details of the construction of the subshift $\Sigma$ in Theorem \ref{thm1}. All involved subshifts are based on a fixed alphabet $\Lambda$ of cardinality $N\geq 2$. 
The desired $\Sigma$ is the intersection of a nested \sq\ of certain subshifts $\Sigma_k$. 
By definition, $\Sigma_k$ consists of all infinite concatenations (and their shifts) of blocks belonging to a family $\G_k\subset\Lambda^{N_{\!k}}$ (each element of $\G_k$ is a block of length $N_{\!k}$). Since $N_k$ tends to infinity, the \tl\ entropy of $\Sigma$ can be computed as the following 
(nonincreasing) limit:
$$
h(\Sigma)=\lim_kh(\Sigma_k) = \lim_k\tfrac1{N_{\!k}}\log(\#\G_k).
$$

The construction begins by setting $N_0=1$ and letting $\G_0=\Lambda$ so that $\Sigma_0$ is the full shift on $N$ symbols.
In the inductive step $k\ge 1$ we assume that the family $\G_{k-1}$ (and thus the subshift $\Sigma_{k-1}$) is already defined and we pass to constructing the family $\G_k$. In order to do so, we must fix several parameters, one of which is a positive integer $m_k$ (called the \emph{multiplier}) equal to the ratio $\frac{N_{\!k}}{N_{\!k-1}}$ (the family $\G_k$ consists of some carefully selected concatenations of $m_k$ blocks from $\G_{k-1}$). The initial value of the sequence $\{m_k\}$ 
is $m_1=M\geq 81$, subsequently the multipliers tend very slowly and nondecreasingly to infinity (the \sq\ $\{m_k\}$ has long intervals of constancy and infinitely many jumps up, each only by a unit). The first index $k$ for which $m_k$ assumes a given value $m\ge M$ is denoted by $K_m$ and called the $m$th \emph{jump index}. The jump indices form a rapidly growing \sq\ whose speed of growth is regulated by specific conditions. Since the rules give only lower bounds, within the same construction scheme we are free to impose any faster growth.

We also make use of two sequences of positive numbers, $\{\epsilon_k\}$ and $\{\delta_k\}$, both tending very slowly to zero (in this paper we will be dealing with only one \sq\ $\{\epsilon_k+\delta_k\}$). We note that in order to verify the uncorrelation between $\Sigma$ and $y$ it suffices to consider only some specific continuous functions $f:\Sigma\to\R$, namely the functions with values in $\{-1,1\}$, depending on finitely many nonnegative coordinates (we call such functions \emph{codes}). A convenient property of each code $f$ is that it can be applied not only to infinite \sq s $x\in\Lambda^\N$ but also to any sufficiently long block $B$ over the alphabet 
$\Lambda$, producing as an output a slightly shorter block $f(B)$ over $\{-1,1\}$ (assuming that $B$ is very long we will ignore the difference in lengths without further consequence). We represent the countable family of all codes as an increasing union of finite families $\F_k$ and in each inductive step $k$ we consider only the codes from $\F_k$.

By definition, the family $\G_k$ consists of all concatenations $B$ of $m_k$ blocks from $\G_{k-1}$ which pass the following \emph{correlation test}: 
\begin{enumerate}
\item[(R)]\label{R} for every $1\le j\le (m_k^2-1)N_{\!k}$ and every $f\in\F_k$, letting $C=y_j^{j+N_{\!k}-1}$ 
we have $|\overline{f(B)C}|<2(\epsilon_k+\delta_k)$. \label{dwa}
\end{enumerate}
Informally, we demand all images of $B$ under the codes from the finite family $\F_k$ to have small correlations with every block of $y$ of length $N_{\!k}$, ending before the position $m_k^2N_{\!k}$. 

\medskip
We denote by  $\gamma_k$  the ``probability of passing the correlation test'', i.e., 
the probability that a block $B\in(\G_{k-1})^{m_k}$ satisfies (R). Then the cardinality of the family $\G_k$ is given by:
$$
\#\G_k = (\#\G_{k-1})^{m_k}\gamma_k.
$$
A composition of the above, applied to $i$ ranging from $1$ to $k$ (recall that $m_i=\tfrac{N_i}{N_{i-1}}$), yields:
$$
\#\G_k = N^{N_{\!k}}\cdot\gamma_1^{\frac{N_{\!k}}{N_1}}\cdot\gamma_2^{\frac{N_{\!k}}{N_2}}\cdots
\gamma_{k-1}^{\frac{N_{\!k}}{N_{\!k-1}}}\cdot\gamma_k^{\frac{N_{\!k}}{N_{\!k}}},
$$
which allows us to evaluate the \tl\ entropy of $\Sigma$ as the limit
$$
h(\Sigma)= \lim_k\frac1{N_{\!k}}\log(\#\G_k)= \log N + \sum_{k=1}^\infty\frac{\log(\gamma_k)}{N_{\!k}}. 
$$
As long as all probabilities $\gamma_k$ are greater than or equal to $\frac12$ (which we later show to be true), we have the following lower estimate of the \tl\ entropy of $\Sigma$:
$$
h(\Sigma)\ge\log N -\log 2\cdot \sum_{k=1}^\infty\frac1{M^k}= \log N-\tfrac{\log2}{M-1}. 
$$
By choosing the initial multiplier $M$ large, the entropy $h(\Sigma)$ can be made as close to $\log N$ as we desire. In this work, we will make use of another immediate consequence of the above entropy formula---an upper bound on the difference between the \tl\ entropies of $\Sigma_p$ and $\Sigma_k$ for any $p<k$\,:
\begin{equation}\label{ent2}
h(\Sigma_p)-h(\Sigma_k)< \log 2\cdot\sum_{i=p+1}^\infty\frac1{N_i}\le \frac{\log 2}{N_p}\cdot\sum_{i=1}^\infty\frac1{(m_p)^i}
=\frac{\log2}{N_p(m_p-1)}.
\end{equation}
\medskip

We will now argue why the uncorrelation between $y$ and $\Sigma$ is uniform. 
\begin{fact}\label{fac}
The subshift $\Sigma$ is uniformly uncorrelated to $y$. 
\end{fact}

\begin{proof}
Of course, it suffices to test uniformity of the uncorrelation only on codes. To this end, we copy the proof of uncorrelation from \cite{DS} and we indicate the place where uniformity is implicit:

${<\!\!\!<}$Let $f$ be any $\{-1,1\}$-valued function depending on finitely many nonnegative coordinates. Fix some point $x\in\Sigma$ and pick $n\in\N$. Let $k$ be the smallest integer such that $n< m^2N_k$ (by convention $m$ abbreviates $m_k$). If $f$ is not in $\mathcal F_k$ then we simply must pick a larger $n$. So, we can assume that $f\in\mathcal F_k$. Now, $x\in\Sigma_k$, which means that $x_1^n$ is a concatenation of the blocks from $\G_k$, except that the first and last component blocks may be incomplete. The contribution of these parts in the length is at most $\frac{2N_k}n$, and since $n\ge m_{k-1}^2N_{k-1}\ge (m-1)^2N_{k-1}>(m-2)N_k$, this contribution is less than $\frac2{m-2}$, and such is also the maximal contribution of these parts in the evaluation of the correlation between $x_1^n$ and $y_1^n$. The rest of the correlation is the average of the correlations of the complete component blocks from $\G_k$ with their respective subblocks of length $N_k$ of $y$. Since all these subblocks end before the position $m^2N_k$, by (R), each of these correlations is less than $2(\epsilon_k+\delta_k)$ in absolute value. Jointly, the absolute value of the correlation of $x_1^n$ with $y_1^n$ does not exceed 
$$
\tfrac 2{m-2}\cdot1 + \tfrac{m-4}{m-2}\cdot2(\epsilon_k+\delta_k).{>\!\!\!>}
$$

Just notice that the above argument is completely independent of $x\in\Sigma$. The final estimate of the correlation depends exclusively on $f$ and $n$ ($f$ determines the lower bound for $n$, and $n$ determines $k$ and hence both $m$, $\epsilon_k$ and $\delta_k$).
\end{proof}

We pass to proving the main result of this paper.

\begin{proof}[Proof of Theorem \ref{thm2}] 
We shall indicate a modification of the construction of $\Sigma$, leading to a new subshift $\Sigma'$ which carries only one invariant measure (and maintains the other properties). Once this is done, obtaining a strictly ergodic (i.e., uniquely ergodic and minimal) subshift uncorrelated to $y$ is trivial: Every minimal subsystem of $\Sigma'$ carries an \im\ and since such measure is unique, the minimal subsystem is also unique (let us denote it by $\Sigma''$) and supports the same measure. Thus, $\Sigma''$ is the desired strictly ergodic subshift. From now on we will focus on constructing $\Sigma'$. The modification consists of two steps. The first one is in fact no modification at all, as we simply impose a faster growth of the \sq\ of jump indices $K_m$. The resulting subshift fits in the original scheme and will be still 
denoted by $\Sigma$. The second modification is more substantial: in the construction of $\Sigma$ we replace 
the families $\G_k$ by their proper subfamilies $\G'_k$. In this manner we create a subsystem $\Sigma'$ of $\Sigma$. 
Clearly, $\Sigma'$ remains uniformly uncorrelated to $y$. We will only need to verify that the entropy of $\Sigma'$ 
is close to $\log N$ and that $\Sigma'$ is strictly ergodic. 

\medskip

We continue by recalling more details of the original construction in \cite{DS}. The speed of growth of $K_m$ is ruled by two conditions:
\begin{enumerate}
\item[(a)] Some technical condition which we will not use or change;
\item[(b)] $9\cdot\alpha(m)\cdot(\frac89)^{K_m-1}<\frac1{2^{m+2}}$, \ where
$\alpha(m)>0$ depends only on $m$. 
\end{enumerate}
\smallskip

For each step number $k$ we define the \emph{reference index} as $p_k=m_k-M$ (which is always smaller than $k$). The reference index grows with $k$ as slowly as $m_k$ does (remaining constant throughout many steps and jumping up only by a unit). Each block from $\G_k$ (which by definition is a concatenation of the blocks from $\G_{k-1}$) is clearly also a concatenation of (much shorter) blocks from $\G_{p_k}$. This fact will soon play an important role in our modified construction.

The proof of the main statement of \cite{DS} relies on the validity of the following lemma
(extracted from Lemma 3.2 in \cite{DS}):
\begin{lem} For every $k$ we have:
\begin{enumerate}
  \item[(A)] $\sum_{s=p_k+1}^{k-1}(1-\gamma_s)<\frac{\delta_k}2$,
  \item[(B)] a technical condition which we will not use or change,
	\item[(C)] $\gamma_k>1-\alpha(m_k)(\tfrac89)^{k-1}$.
\end{enumerate}
\end{lem}
Note also that {\rm(C)} combined with {\rm(b)} and with the obvious fact that $k\ge K_{m_k}$ guarantees that $\gamma_k>1-2^{-(m_k+2)}$ (so $\gamma_k$ is much larger than $\frac12$). The proof of the above lemma starts with showing (A) using the condition (b) and the inductively assumed, for all $s<k$, condition (C). The proof of the condition (B) depends only on (A), and then the proof of (C) (the version for $\gamma_k$) relies on (A) and (B). Since we do not invoke (B), it is essential for us that (C) follows from (A). This implication relies on the particular design of the correlation test (R).

\medskip
We now introduce the first innovation in the construction---a slightly sharper requirement on how large $K_m$ must be. Namely, instead of (b) we demand that
\begin{equation}
9\cdot\alpha(m)\cdot(\tfrac{8.5}9)^{K_m-1}<\tfrac1{2^{m+2}}.\tag{b'}
\end{equation}
Let us discuss some consequences of this modification. Suppose that for all $s<k$ we replace the families $\G_s$ by their subfamilies $\G'_s$ in such a way that the corresponding probabilities $\gamma_s'=\frac{\#\G'_s}{(\#\G'_{s-1})^{m_s}}$ instead of (C) (for $s$) satisfy a slightly weaker version:
\begin{equation}
\gamma'_s>1-\alpha(m_s)(\tfrac{8.5}9)^{s-1}.\tag{C'}
\end{equation}
Then, as easily verified, our sharpened condition (b') allows to prove a version of (A) 
with these new probabilities replacing the old ones, as follows:
\begin{equation}
\sum_{s=p_k+1}^{k-1}(1-\gamma'_s)<\tfrac{\delta_k}2.\tag{A'}
\end{equation}
Now, if we denote by $\bar\gamma_k$ the probability that a concatenation from $(\G'_{k-1})^{m_k}$ passes the correlation test (R), then, by the same proof as that of (A)$\implies$(C), (A') implies (C) for $\bar\gamma_k$, i.e.,  
\begin{equation}
\bar\gamma_k>1-\alpha(m_k)(\tfrac89)^{k-1}.\tag{$\bar{\text C}$}
\end{equation} (The reason why we do not denote $\bar\gamma_k$ by $\gamma'_k$ will become clear later.)

\medskip

We need to impose one more requirement on the growth of the jump indices $K_m$. 
First we fix a decreasing to zero \sq\ of positive numbers $\{r(j)\}_{j\ge1}$, and for each $j$ we find (by referring to the comment following \eqref{metr}) $n(j)\in\N$ and $\theta(j)>0$ such that if two \im s $\mu$ and $\nu$ satisfy 
$$
|\mu(D)-\nu(D)|<\theta(j)
$$
for all blocks (identified with cylinders) $D$ of length $n(j)$ then we have $d(\mu,\nu)<r(j)$. Now, for each $j$ we set 
$$
\beta(j)=\frac{(\theta(j))^2}{128}
$$
and we find the smallest integer $p(j)$ such that 
\begin{equation}\label{ent1}
\frac{\log 2}{m_{p(j)}-1}<\beta(j) \text{ \ \ and \ \ } \frac{n(j)}{N_{p(j)}}<\frac{\theta(j)}4.
\end{equation}
The new requirement (on the largeness of $K_m$) which we are about to force applies only to indices $m=m(j)=p(j)+M$, $j\ge 1$. Since throughout this and the following paragraph $j$ remains fixed, we will skip ``$(j)$'' in the denotation of $m(j)$, $n(j)$, $p(j)$, etc. Note that between steps $p$ and $K_m$, the multiplier is at least $m_p$ thus the ratio $\frac{N_{K_m}}{N_p}$ will be larger than or equal to $m_p^{K_m-p}$. Finally, recall that $N$ stands for the cardinality of the alphabet~$\Lambda$.

Here is the requirement: we demand that $K_m$ is so large that
\begin{equation}
2N^{n}\exp(-\beta m_p^{K_m-p})
	<\alpha(m)\bigl((\tfrac{8.5}9)^{K_m-1}-(\tfrac89)^{K_m-1}\bigr).\tag{E}
\end{equation}
This is clearly satisfied if $K_m$ is large enough, because the left hand side decreases (with growing $K_m$) to zero with doubly exponential speed, while the right hand side does it only exponentially fast. We remark, that since the values of $m_p$ and $m$ depend on 
jump steps much smaller than $K_m$, this inductive definition of the sequence $\{K_m\}$ is correct.

\medskip
With the above choice of the jump steps, the otherwise unmodified scheme leads to a \sq\ of families $\G_k$, a \sq\ of subshifts $\Sigma_k$ and a subshift $\Sigma$ which has entropy larger than $\log N - \frac{\log 2}{M-1}$ and is uniformly uncorrelated to $y$.
We will now present the second, more substantial, modification of the inductive construction through which we create subfamilies $\G'_k$ of $\G_k$, subsystems $\Sigma_k'$ of $\Sigma_k$ and the desired subsystem $\Sigma'$ of $\Sigma$. The modification consists in applying, 
in addition to the correlation test (R), a new test, which we will call the Bernstein's test (described below). The new test will be applied only at steps whose indices $k$ have the form $K_{m(j)}$, $j\ge 1$. Nevertheless, the modification defines (indirectly) also the subfamilies $\G'_k\subset\G_k$ for other indices $k$.
 
We start by not changing $\G_s$ (i.e., letting $\G'_s=\G_s$) for all $s<K_{m(1)}$. 
Suppose that for some $j\ge 1$ we have defined the subfamilies $\G'_s\subset\G_s$ for all $s<K_{m(j)}$, and that the corresponding probabilities $\gamma'_s=\frac{\#\G'_s}{(\#\G'_{s-1})^{m_s}}$ satisfy the inequality (C'). Since $j$ is now fixed we return to our previous notational convention. Moreover, we will abbreviate $K_m$ as $k$ (and so $m=m_k$). The goal is now to define $\G'_k\subset\G_k$.

At this point we create a temporary family $\bar\G_k$ by applying to the concatenations from $(\G'_{k-1})^m$ just the correlation test (R). By the discussion on the preceding page, the probability $\bar\gamma_k$ of passing the test satisfies the condition ($\bar{\text C}$). For brevity we denote by $q$ the ratio $\frac{N_k}{N_p}$.

For a fixed block $D\in\Lambda^{n}$, let us define a random variable $\mathsf X_D$ on $\G'_p$ (note that in step $k$, $\G'_p$ is already defined)
by setting $\mathsf X_D(b)=\mathsf{freq}_b(D),~b\in\G'_p$, and let $\bar X_D$ denote the expected value of $X_D$.

Now we create the target family $\G'_k$ by discarding from $\bar\G_k$ all blocks $B$ 
which fail the \emph{Bernstein's test}, i.e., have the following property (since $\bar \G_k\subset(\G'_p)^q$ we can formulate the condition for all blocks in $(\G'_p)^q$):

\begin{enumerate}
	\item[(F)]  A block $B\in(\G'_p)^q$,\ $B=b_1b_2\dots b_q,\ b_i\in\G'_p, i=1,\ldots q$,~ fails the Bernstein's test if
	for at least one block $D\in\Lambda^{n}$, we have
	$$
	\Bigl|\frac1q\sum_{i=1}^q \mathsf X_D(b_i)-\bar{\mathsf X}_D\Bigr|>\sqrt{8\beta}.
	$$
\end{enumerate}

This completes the definition of $\G'_k$ for $k=K_{m(j)}$. 
We will show in a moment that the corresponding probability $\gamma_k'=\frac{\#\G'_k}{(\#\G'_{k-1})^m}$ satisfies (C') (the version for $k$). 
For indices $k$ strictly between $K_{m(j)}$ and $K_{m(j+1)}$ the families $\G'_k$ are created according to the original scheme, i.e., using the correlation test only.
By the same argument as used above for $\bar\gamma_k$, the resulting probabilities $\gamma'_k$ satisfy (C) (and hence (C') for $k$), so that the construction can be continued for $k=K_{m(j+1)}$.

Modulo the missing proof of (C') for $k=K_{m(j)}$, the families $\G'_k$ and the resulting subshift $\Sigma'$ are now determined. Once (C') is proved, we will also know that all the probabilites $\gamma'_k$ are larger than $\frac12$, which will in turn imply that
the subshift $\Sigma'$ has \tl\ entropy at least $\log N-\frac{\log 2}{M-1}$.
In the end we will also need to verify that $\Sigma'$ is strictly ergodic.
\medskip

We pass to proving the missing condition (C') for $k=K_{m(j)}$. Since $j$ is now fixed,  we apply again our notational convention in which ``$(j)$'' is skipped and $K_m=k$.
The classical Bernstein's inequality (see e.g. \cite{B}) 
implies that in the space $(\G'_p)^q$ the probability of failing the Bernstein's test is smaller than $2N^{n}\exp(-2q\beta)$. We can express this fact as follows:
\begin{equation}\label{pro1}
\P(\text{F})<2N^{n}\exp(-2q\beta),
\end{equation} 
where $\P$ denotes the normalized counting measure on $(\G'_p)^q$ and F symbolizes 
the event of failing the Bernstein's test.

Note that, since all the new probabilities $\gamma'_s$ and $\bar\gamma_k$ are larger than $\frac12$, the inequality \eqref{ent2} is valid for $\Sigma'_p$ and $\bar\Sigma_k$:
$$
h(\Sigma'_p)-h(\bar\Sigma_k)<\frac{\log2}{N_p(m_p-1)}.
$$
By the definition of $p=p(j)$, the right hand side above is less than $\frac\beta{N_p}$. We have proved that 
$$
\tfrac1{N_k}\log\#\bar\G_k>\tfrac1{N_p}\log\#\G'_p-\tfrac\beta{N_p},
$$
which implies that
$$
\#\bar\G_k > (\#\G'_p)^{q}\exp(-q\beta).
$$
In other words, the probability that a randomly chosen concatenation $B\in(\G'_p)^q$ belongs to $\bar\G_k$ is larger than $\exp(-q\beta)$. 
Let us note this fact as follows:
\begin{equation}\label{pro2}
\P(\bar\G_k)>\exp(-q\beta).
\end{equation} 
Dividing the right hand side of \eqref{pro1} by the right hand side of \eqref{pro2} we obtain an upper estimate on the conditional probability in $\bar\G_k$ of failing the Bernstein's test:
$$
\P(\text{F}|\bar\G_k)< 2N^{n}\exp(-q\beta).
$$
Since $q\ge m_p^{k-p}$, the right hand side above is dominated by the left hand side of the inequality (E) (in which $K_m$ is written as $k$), hence we obtain:
$$
\P(\text{F}|\bar\G_k)< \alpha(m)\bigl((\tfrac{8.5}9)^{k-1}-(\tfrac89)^{k-1}\bigr).
$$

So, the new probability $\gamma_k'$ that a block from $(\G'_{k-1})^m$ passes both tests (and thus belongs in $\G'_k$) equals
\begin{multline*}
\gamma'_k=\bar\gamma_k\cdot(1-\P(\text{F}|\bar\G_k))>\\
(1-\alpha(m)(\tfrac89)^{k-1})\Bigl(1-\alpha(m)\bigl((\tfrac{8.5}9)^{k-1}-(\tfrac89)^{k-1}\bigr)\Bigr)>\\ 1-\alpha(m)(\tfrac{8.5}9)^{k-1}.
\end{multline*} 
We have proved that (C') holds for $k$, as needed.
\medskip

It remains to show strict ergodicity of $\Sigma'$. Since the sets $\M(\Sigma'_k)$ 
form a nested \sq\ of compact subsets of $\M(\Lambda^\N)$ with intersection equal to $\M(\Sigma')$ it suffices to prove that the diameters of $\M(\Sigma'_k)$ tend to zero along some subsequence. Once again, we fix some index $j$ and return to our notational convention
as above. Recall that if two \im s $\mu$ and $\nu$ satisfy 
$$
|\mu(D)-\nu(D)|<\theta
$$
for all blocks $D$ of length $n$, then we have $d(\mu,\nu)<r$. 
Let us call blocks of length $n$ \emph{short}, the elements of $\G'_p$ \emph{medium blocks} (these have length $N_p$), and the concatenations belonging to $\G'_k\subset (\G'_p)^q$ \emph{long blocks} (as before, $q=\frac{N_k}{N_p}$).

If $D$ is a short block and $B$ is a long block, we can approximate the frequency of 
occurrence of $D$ in $B$ as the weighted average of its frequencies in the medium blocks $b\in\G'_p$, as follows:
$$
\mathsf{freq}_B(D)=\Bigl[\sum_{b\in\G'_p}\mathfrak{freq}_B(b)\cdot\mathsf{freq}_b(D)\Bigr]\pm\tfrac n{N_p}.
$$
The term $\mathfrak{freq}_B(b)$ accounts only ``regular'' occurrences of $b$ in $B$, i.e., the occurrences of $b$ as components in the concatenation of elements of $\G'_p$ which constitute $B$. The symbol $\pm\tfrac n{N_p}$ stands for an error term whose absolute value does not exceed $\tfrac n{N_p}$ (and thus $\frac\theta4$). It is needed to account the possible occurrences of $D$ covered by two components of this concatenation. 
The sum in square brackets equals $\frac1q\sum_{i=1}^q \mathsf X_D(b_i)$ (using the notation introduced in the paragraph preceding (F)). 
Since $B$ belongs to $\G_k'$ and hence passes the Bernstein's test, 
this average differs from $\bar X_D$ by less than $\sqrt{8\beta}$ which equals $\frac{\theta}{4}$. 
Eventually, we have
$$
\mathsf{freq}_B(D)=\bar{\mathsf X}_D\pm\tfrac{\theta}{2}.
$$
This clearly implies that the frequencies of $D$ in any two blocks $B,B'\in\G'_k$ differ from each-other by at most $\theta$. This property applies also to all elements $x\in\Sigma'_k$: for any $x,x'\in\Sigma'_k$ the lower density of occurrence of $D$ in $x$ 
and the upper density of occurrence of $D$ in $x'$ differ by at most $\theta$ (the error term $\pm\frac n{N_p}$ absorbs also these occurrences of $D$ in $x\in\Sigma'_k$ which are covered by two components of the infinite concatenation of the blocks $B\in\G'_k$ which constitute $x$). This easily implies that $|\mu(D)-\nu(D)|<\theta$ for any $\mu,\nu\in\M(\Sigma'_k)$. By the choice of $n$ and $\theta$ we obtain that the diameter of $\M(\Sigma'_k)$ is at most $r=r(j)$. Since the \sq\ $\{r(j)\}$ tends to zero, we have shown that the diameters of $\M(\Sigma'_k)$ tend to zero along the subsequence indexed by $k=K_{m(j)}$, $j\ge 1$. This ends the proof.
\end{proof}

\end{document}